\documentclass[a4paper,11pt]{article}

\usepackage{makeidx}
\usepackage{amscd}
\usepackage{amsmath}
\usepackage{amssymb}
\usepackage{amsthm}

\usepackage[all]{xy}
\theoremstyle{plain}
\newtheorem{thm}{Theorem}[section]

\newtheorem{lem}[thm]{Lemma}
\newtheorem{cla}[thm]{Claim}

\newtheorem{prop}[thm]{Proposition}

\newtheorem{prob}[thm]{Problem}

\theoremstyle{definition}
\newtheorem{defn}[thm]{Definition}
\newtheorem{rem}[thm]{Remark}

\newtheorem{defn-prop}[thm]{Definition-Proposition}

\newcommand{\PP}{\mathbb P}
\newcommand{\Z}{\mathbb Z}

\newcommand{\C}{\mathbb C}

\newcommand{\Ker}{\mathop{\mathrm{Ker}}\nolimits}

\newcommand{\id}{\ensuremath{\mathop{\mathrm{id}}}}

\newcommand{\Pic}{\mathop{\mathrm{Pic}}\nolimits}

\newcommand{\Lotimes}{\stackrel{\mb L}{\otimes}}

\newcommand{\SL}{\operatorname{SL}}

\newcommand{\Auteq}{\operatorname{Auteq}}
\newcommand{\module}{\operatorname{mod}}
\newcommand{\Aut}{\operatorname{Aut}}

\newcommand{\str}{\mathcal{O}} 
\newcommand{\mc}{\mathcal}
\newcommand{\mb}{\mathbb}

\newcommand{\Supp}{\ensuremath{\operatorname{Supp}}}
\renewcommand{\labelenumi}{(\roman{enumi})}

\newcommand{\ord}{\ensuremath{\operatorname{ord}}}
\newcommand{\FM}{\ensuremath{\operatorname{FM}}}

\newcommand{\Span}[1]{\left<#1\right>}

\title{Fourier--Mukai partners of elliptic ruled surfaces}
\author{Hokuto Uehara}
\date{}
\begin{document}
\maketitle
\begin{abstract}
We study Fourier--Mukai partners of elliptic ruled surfaces. We also describe the autoequivalence group of the derived categories of 
ruled surfaces with an elliptic fibration, by using \cite{Ue15}.
\end{abstract}


\section{Introduction}
\subsection{Motivations and results}\label{subsec:motivation}
Let $X$ be a smooth projective variety over $\C$ and $D(X)$ the bounded derived category of coherent sheaves on $X$.
If $X$ and $Y$ are smooth projective varieties with equivalent derived categories, then we call
$X$ and $Y$ \emph{Fourier--Mukai partners}.
We denote by $\FM (S)$ the set of isomorphism classes of Fourier--Mukai partner of $X$:
$$
\FM (X):=\{Y\text{ smooth projective varieties }\mid 
D(X)\cong D(Y) \}/\cong.
$$
It is an interesting problem to determine the set 
$\FM (X)$ for a given $X$.
There are several known results in this direction.  
For example, Bondal and Orlov show that if $K_X$ or $-K_X$ is ample, then $X$ can be entirely reconstructed from $D(X)$, 
namely $\FM (X)=\{X\}$ (\cite{BO95}).
To the contrary, there are examples of non-isomorphic varieties $X$ and $Y$ having equivalent derived categories. 
For example, in dimension $2$, if $\FM (X)\ne \{X\}$, then $X$ is a K3 surface, an abelian surface or 
a relatively minimal elliptic surface with non-zero Kodaira dimension (\cite{BM01}, \cite{Ka02}).

By the classification of surfaces, relatively minimal elliptic surfaces with negative Kodaira dimension 
are either rational elliptic surfaces or elliptic ruled surfaces.
In \cite{Ue04, Ue11}, the author studies the set $\FM(S)$ of rational elliptic surfaces $S$.
In this paper, we describe the set $\FM(S)$ of elliptic ruled surfaces $S$:


\begin{thm}\label{thm:FMpartner}
Let $f\colon S=\PP (\mc E)\to E$ be a $\PP^1$-bundle
 over an elliptic curve $E$, and $\mc E$ be a normalized locally free sheaf of  rank $2$.
If $|\FM (S)|\ne 1$, there is a degree $0$ line bundle
$\mc L\in \hat{E}:=\Pic ^0 E$  of order $m>4$ such that $\mc E= \mc O_{E}\oplus \mc L$. 
Furthermore in this case, we have
$$
\FM (S)=\{\PP (\mc{O}_E\oplus \mc{L}^i)\mid i\in (\Z/m\Z)^*\}/\cong.
$$
This set consists of $\varphi(m)/|H_{\hat{E}}^{\mc{L}}|$ elements.
Here, $\varphi$ is the Euler function, and 
 $H_{\hat{E}}^{\mc{L}}$ is a group defined in \S \ref{subsec:technical} with $|H_{\hat{E}}^{\mc{L}}|=2,4$ or $6$, 
depending on the choice of $E$ and $\mc{L}$. 
\end{thm}

As an application, in \S \ref{sec:autoeq}, we describe the autoequivalence group 
of the derived categories of certain elliptic ruled surfaces by using the result in \cite{Ue15}.

\subsection{Notation and conventions}\label{subsec:notation_convention}
All varieties will be defined over $\C$, unless stated otherwise.
A \emph{point} on a variety will always mean a closed point.
By an \emph{elliptic surface}, we will
always mean a smooth projective surface $S$ together with a smooth projective
curve $C$ and a relatively minimal morphism $\pi\colon S\to C$ whose general fiber is an elliptic curve. Here  a \emph{relatively
minimal morphism} means a morphism whose fibers contains no $(-1)$-curves. Such a morphism $\pi$ is 
called an \emph{elliptic fibration}.

For an elliptic curve $E$ and some positive integer $m$, we denote  the set of points of order $m$ by ${}_mE$.  
Furthermore, we denote the dual elliptic curve, namely the group scheme $\Pic ^0 E$ of line bundles on $E$ of degree $0$,
by $\hat{E}$, and the group of automorphisms of $E$ fixing the origin by $\Aut_0 E$.
 
$D(X)$ denotes the bounded derived category of coherent sheaves on an algebraic variety $X$, and 
$\Auteq D(X)$ denotes the group of isomorphism classes of $\C$-linear exact autoequivalences of a $\C$-linear triangulated category $D(X)$.

Let $X$ and $Y$ be smooth projective varieties.
For an
object $\mathcal{P}\in D(X\times Y)$, we define an exact functor $\Phi^{\mathcal{P}}$, 
called an \emph{integral functor}, to be
$$
\Phi^{\mathcal{P}}:= 
\mathbb{R}p_{Y*}(\mathcal{P}\Lotimes p^{*}_X(-))\colon D(X)\to D(Y),
$$
where we denote the projections by $p_X\colon X\times Y\to X$ and $p_Y\colon X\times Y\to Y$.
By the result of Orlov (\cite{Or97}),
for a fully faithful functor $\Phi\colon D(X)\to D(Y)$,
there is an object $\mathcal{P}\in D(X\times Y)$, unique up to isomorphism, such that 
$
\Phi\cong \Phi^{\mathcal{P}}.
$
If an integral functor $\Phi^{\mc{P}}$ is an equivalence,
it is called a \emph{Fourier--Mukai transform}.

\subsection{Acknowledgments} 
The author is supported by the Grants-in-Aid 
for Scientific Research (No.23340011). 
Main part of this paper was written during his staying 
at the  Max-Planck Institute for Mathematics, 
in the period from April to September 2014.
He appreciates the hospitality. 


\section{Preliminaries}\label{sec:preliminaries}

\subsection{Fourier--Mukai transforms on elliptic surfaces}\label{subsec:bridgeland}

 Bridgeland, Maciocia and Kawamata show in \cite{BM01,Ka02}
that if a smooth projective surface $S$ has a non-trivial Fourier--Mukai partner $T$, that is $|\FM (S)|\ne 1$,
then both of $S$ and $T$ are abelian varieties, K3 surfaces or elliptic surfaces with 
non-zero Kodaira dimension. 
We consider the last case in more detail. Many results in this subsection are shown in \cite{Br98}.
Readers are recommended to refer to the original paper \cite{Br98}.

Let $\pi:S\to C$ be an elliptic surface. 
For an object $E$ of $D(S)$, we define the fiber degree of $E$ as
\[d(E)=c_1(E)\cdot F, \]
where $F$ is a general fiber of $\pi$. 
Let us denote by $r(E)$ the rank of $E$ and by $\lambda_{S}$  
the highest common factor of the fiber degrees of objects of $D(S)$. 
Equivalently,
$\lambda_{S}$ is the smallest number $d$ such that there is a  $d$-section of $\pi$. 
Consider integers $a$ and $b$ with $a>0$ and $b$ coprime to $a\lambda_{S}$. 
Then, there exists a smooth,
$2$-dimensional component $J_S (a,b)$ of the moduli space of pure dimension 
one stable sheaves on $S$,
the general point of which represents a rank $a$, degree $b$ stable 
vector bundle supported on a smooth fiber of $\pi$. 
There is a natural morphism $J_S (a,b)\to C$, taking a point representing
 a sheaf supported on the
fiber $\pi ^{-1}(x)$ of $S$ to the point $x$. This morphism is a relatively minimal 
elliptic fibration. Furthermore, there is a universal sheaf on $\mc{U}$ on 
$J_S(a,b)\times S$ such that the integral functor $\Phi^\mc{U}$ is a Fourier--Mukai transform.

Put $J_S(b):=J_S(1,b)$. 
Obviously, 
we have $J_S(1)\cong S$. 
As is shown in \cite[Lemma 4.2]{BM01}, there is also an isomorphism
\begin{equation*}\label{eqn:J(a,b)=J(b)}
J_S(a,b)\cong J_S(b).
\end{equation*} 


\begin{thm}[Proposition 4.4 in \cite{BM01}]\label{BMelliptic}
Let $\pi :S\to C$ be an elliptic surface 
and $T$ a smooth projective variety.
Assume that the Kodaira dimension $\kappa (S)$ is non-zero.
Then the following are equivalent.
\renewcommand{\labelenumi}{(\roman{enumi})}
\begin{enumerate}
\item 
$T$ is a Fourier--Mukai partner of $S$. 
\item 
$T$ is isomorphic to $J_S(b)$ for some integer $b$ with $(b,\lambda _{S})=1$. 
\end{enumerate}
\end{thm}
There are natural isomorphisms
\begin{equation}\label{eqn:JSB}
J_S(b)\cong J_S(b+\lambda_{S})\cong J_S(-b)
\end{equation}
(see \cite[Remark 4.5]{BM01}).
Therefore, we can define the subset 
\begin{equation*}
H_S:=\{b\in (\Z/\lambda_{S}\Z)^* \mid J_S(b)\cong S \}
\end{equation*}
of the multiplicative group $(\Z/\lambda_{S}\Z)^*$.
We can  see that $H_S$ is a subgroup of $(\Z/\lambda_{S}\Z)^*$,
and there is a natural 
one-to-one correspondence between 
the set $\FM (S)$ and the quotient group 
$(\Z/\lambda_{S}\Z)^* /H_S$ (see  \cite[\S 2.6]{Ue15}). 

\begin{cla}\label{cla:lambda}
When $\lambda_S\le 4$, we have $|\FM (S)|=1$.
\end{cla}

\begin{proof}
When $\lambda_S\le 2$, $(\Z/\lambda_{S}\Z)^*$ is 
trivial and hence, $\FM (S)=\{S\}$. 
 For $\lambda_S>2$ and $b\in (\Z/\lambda_{S}\Z)^*$, we have $b\ne \lambda_S-b$ 
in $(\Z/\lambda_{S}\Z)^*$, 
and hence, the isomorphisms \eqref{eqn:JSB} yield
\begin{equation*}
|\FM (S)|\le \varphi (\lambda_{S})/2,
\end{equation*}
where $\varphi$ is the Euler function. This inequality implies
$|\FM (S)|=1$ for $\lambda_S\le 4$.
\end{proof}

In general, it is not easy  to describe the group $H_S$,
equivalently to describe the set $\FM(S)$, concretely. 
However, even if $\lambda_S\ge 5$, there are several examples in which 
we can compute the cardinality of  the set $\FM (S)$  (see \cite[Example 2.6]{Ue11}).

\subsection{Some technical lemmas on elliptic curves}\label{subsec:technical}

Let $F$ be an elliptic curve. 
For points $x_1,x_2\in F$, to distinguish the summations as divisors 
and as elements in the group scheme $F$,
we denote by $x_1\oplus x_2$ 
the sum 
of them by the operation of $F$,
and 
$$
i\cdot x_1:=x_1 \oplus \cdots \oplus x_1 \quad \mbox{($i$ times)}.
$$
We also denote by 
$$
i x_1:=x_1 + \cdots + x_1 \quad \mbox{($i$ times)}
$$ the divisor on $F$ of degree $i$.
As is well-known, there is a group scheme isomorphism
\begin{equation}\label{eqn:dual}
F\to \hat{F} \qquad x\mapsto \mc O_F(x-O),
\end{equation}
where $O$ is the origin of $F$. If we identify $\hat{F}$ and $F$ by \eqref{eqn:dual},                                                                               
so called \emph{the normalized Poincare bundle}  $\mc P_0$ on $F\times F$ is defined by
$$
\mc{P}_0:=\mc{O}_{F\times F}(\Delta_F-F\times O-O\times F),
$$
where $\Delta_F$ is the diagonal of $F$ in $F\times F$.
It satisfies that 
$$
\mc{P}_0|_{F\times x}\cong \mc{P}_0|_{x\times F}\cong\mc O_F(x-O)
$$
for a point $x\in F$.

Let us fix  an element $a\in{}_mF$ with a positive integer $m$. 
Let us denote by $E$ the quotient variety $F/\Span{a}$, by 
$$q\colon F\to E$$ 
the quotient morphism, and by 
$$
\hat{q}\colon \hat{E}\to\hat{F}   
$$
the dual isogeny of $q$.
Define a subgroup of $(\Z/m\Z)^*$ as
$$
H_F^a:=\{ k\in(\Z/m\Z)^*\mid \text{ $\exists\phi\in \Aut_0 F$ such that $\phi(a)=k\cdot a$} \}.
$$ 
 
Recall that 
%
\begin{itemize}
\item
$F\cong\C/(\Z+\sqrt{-1}\Z)$, $\Aut_0 F=\{ \pm 1,  \pm\sqrt{-1}\}$  when $j(F)=1728$,
\item
$F\cong\C/(\Z+\omega\Z)$, $\Aut_0 F=\{\pm 1,\pm\omega,\pm\omega^2 \}$ when  $j(F)=0,$ and 
\item
$\Aut_0 F=\{\pm 1 \}$ when $j(F)\ne 0,1728$. 
\end{itemize}
%
Here $j(F)$ is the $j$-invariant of $F$, and we put $\omega=\frac{-1+\sqrt{-3}}{2}$.
We use the following technical lemmas in the proof of Theorem \ref{thm:FMpartner}.


\begin{lem}\label{lem:j-invariant}
Suppose that $m>3$. Then
exactly one of the following three cases for $F$ and $a\in{}_mF$ occurs.
\begin{enumerate}
\item
The equality $H_F^a=\{\pm 1\}$ holds.
\item 
We have $j(F)=1728$, and there is an integer $n$ such that $m$ divides $n^2+1$. (Note that this condition implies that $\pm n\in (\Z/m\Z)^*$.) Moreover,
the point $a\in F$ is an element in the subgroup 
$$
\Span{\frac{n}{m}+\frac{1}{m}\sqrt{-1}}
$$ 
of $F\cong \C/(\Z+\sqrt{-1}\Z)$, and the equality $H_F^a=\{\pm 1,\pm n \}$ holds.
\item 
We have $j(F)=0$, and there is an integer $n$ such that $m$ divides $n^2+n+1$. (Note that this condition implies that $\pm n\in (\Z/m\Z)^*$.) Moreover, 
the point $a\in F$ is an element in the subgroup 
$$
\Span{\frac{n+1}{m}+\frac{1}{m}\omega}
$$ 
of $F\cong \C/(\Z+\omega\Z)$, and the equality
$H_F^a=\{\pm 1,\pm n,\pm n^2 \}$ holds.
\end{enumerate}
\end{lem}

\begin{proof}
When $j(F)\ne 0,1728$, obviously the case (i) occurs.

Next, let us consider the case  $j(F)=1728$. 
Put $a=\frac{x}{m}+\frac{y}{m}\sqrt{-1}$  for some $x,y\in \Z$,
 and suppose first that an equality
\begin{equation}\label{eqn:aka}
\sqrt{-1}a=n\cdot a
\end{equation}
holds for some $n\in \Z$. Then we have
\begin{equation}\label{eq:kxy}
nx\equiv -y,\quad ny\equiv x \quad(\module m).
\end{equation}
Hence, we know that $a=\frac{ny}{m}+\frac{y}{m}\sqrt{-1}$, and since the order of $a$ in $F$ is $m$,
$m$ and $y$ are coprime. The coprimality and the equations \eqref{eq:kxy}
yield that $m$ divides $n^2+1$.
The coprimality also implies that 
the subgroups $\Span{a}$ and $\Span{\frac{n}{m}+\frac{1}{m}\sqrt{-1}}$ coincide. 
We know  from  $\Aut_0 F=\{ \pm 1,  \pm\sqrt{-1}\}$ that 
$H_F^a=\{\pm 1,\pm n \}$ holds.


In the case (iii), the proof is similar.

It follows from the conditions on $m$ and $n$  that $|H_F^a|=2,4$ and $6$ in the case (i), (ii) and (iii) respectively, hence
the two cases do not occur at the same time. 
\end{proof}

Recall that 
\begin{align*}
H_{\hat{E}}^\mc{L}:=&\{ k\in(\Z/m\Z)^*\mid \text{ $\exists\hat{\psi}\in \Aut_0 \hat{E}$ such that $\hat{\psi}(\mc{L})=\mc{L}^k$} \}\\
=&\{ k\in(\Z/m\Z)^*\mid \text{ $\exists\psi\in \Aut_0 E$ such that $\psi^*\mc{L}=\mc{L}^k$} \}
\end{align*}
for a line bundle $\mc{L}\in{}_m\hat{E}$. 


\begin{lem}\label{lem:F=E}
In each case of Lemma \ref{lem:j-invariant}, the equality
$
H^a_F=H^{\mc{L}}_{\hat{E}}$ 
holds for any
$\mc{L}\in{}_m\hat{E}$ with $\ker \hat{q}=\Span{\mc{L}}$.
(In particular, there is an isomorphism $F\cong E$  in the cases (ii) and (iii), since their $j$-invariants coincide.) 
\end{lem}

\begin{proof}
Let us consider the case (ii) first.
Let $L$ be the lattice generated by $1$ and $\sqrt{-1}$ in $\C$ so that $F$ with $j(F)=1728$ is isomorphic to $\C/L$. 
Moreover, the elliptic curve $E=F/\Span{a}$ is isomorphic to $\C/(L+\Span{a})$. 
We can see that the lattice $L+\Span{a}$ is preserved by the complex multiplication by $\sqrt{-1}$.
(Hence, $j(E)=1728$, which implies $F\cong E$.) 
It turns out that the quotient morphism 
$$
q\colon F\cong\C/L\to E\cong\C/(L+\Span{a})
$$ 
induced by the inclusion $L\hookrightarrow L+\Span{a}$
 is compatible with the complex multiplication by $\sqrt{-1}$.

Take an element $\frac{1}{m}\in \C/L(\cong F)$, and put 
$$
a:=\frac{ny}{m}+\frac{y}{m}\sqrt{-1}
$$
for the integer $n$ in (ii) and some $y\in (\Z/m\Z)^*$. 
We define $\mc{L}'$ to be the element in $\hat{E}$ corresponding to $q\bigl(\frac{1}{m}\bigr)\in E$ via $E\cong \hat{E}$. 
Then we have
$$
\sqrt{-1}q\bigl(\frac{1}{m}\bigr)=q\bigl(\frac{1}{m}\sqrt{-1}\bigr)=q\bigl(y^{-1}a-\frac{n}{m}\bigr)=-nq\bigl(\frac{1}{m}\bigr),
$$
and this implies the equality $H^a_F=\{\pm1,\pm n\}=H^{\mc{L}'}_{\hat{E}}$.
We can also see that 
$$
\Span{a,\frac{1}{m}}=\Span{\frac{\sqrt{-1}}{m},\frac{1}{m}}=\ker [m],
$$
where $[m]$ is the multiplication map by $m$. 
Recall that $[m]=\hat{q}\circ q$ and $\ker q=\Span{a}$. 
Consequently,  we have $\ker \hat{q}=\Span{\mc{L}'}$.
For any $\mc{L}\in{}_m\hat{E}$ with $\ker \hat{q}=\Span{\mc{L}}$, the equality 
$H^{\mc{L}}_{\hat{E}}=H^{\mc{L}'}_{\hat{E}}$ holds, which gives 
 the assertion.

The proof of the case (iii) is similar. 

Next let us take an element $\mc{L}\in \ker\hat{q}$, and suppose that $|H_{\hat{E}}^{\mc{L}}|=4$ or $6$, namely
the case (ii) or (iii) occurs for $\hat{E}$ and $\mc{L}\in{}_m\hat{E}$. Then we have already shown above that 
$H^a_F=H^{\mc{L}}_{\hat{E}}$ (just by replacing the roles of $\hat{E}$ and $F$). Consequently, in the case (i), 
we again obtain the assertion.
\end{proof}

\subsection{Elliptic ruled surfaces over a field of arbitrary characteristic}\label{subsec:positive}
In this subsection, we refer a result which is needed in the proof of Theorem \ref{thm:FMpartner}. 
The results and notation here over a positive characteristic field are not logically needed in this paper, 
but we leave them to explain a background of Problem \ref{prob:positive}. 

Let $k$ be an algebraically closed field of characteristic $p \geq 0$.
%
Suppose that $E$ is an elliptic curve defined over $k$,   $\mc E$ is a normalized,  in the sense of \cite[V.~\S 2]{Ha77}, 
locally free sheaf of rank $2$ on $E$, and 
$f \colon S = \mathbb{P} (\mathcal{E} ) \rightarrow E$ 
is a $\PP ^1$-bundle on $E$. 
Set $e := -\deg \mathcal{E}$. Then 
we can see that $e=0$ or $-1$
if $-K_S$ is nef, and in particular, if $S$ has an elliptic fibration $\pi \colon S\to \PP ^1$.
Furthermore, when the locally free sheaf $\mc E$ is decomposable and $e=0$,
it turns out that  $\mc E=\mc O_E\oplus \mc L$ for some $\mc L\in \hat{E}$.
When $e=-1$, it is indecomposable (see \cite[V.~Theorem 2.12]{Ha77}).  

We use the following result to 
describe the set $\FM (S)$ for elliptic ruled surfaces $S$ in Theorem \ref{thm:FMpartner}.


\begin{thm}[\cite{To11}]\label{thm:TU14}
We use the above notation. 
\begin{enumerate}
\item
For $e = 0$, $S$ has an elliptic fibration in the cases (i-1), (i-2) and (i-5).
Moreover, we have the following:
\begin{center}
\begin{tabular}{|c|c|c|c|}
\hline
       & $\mathcal{E}$ & singular fibers   & $p$ \\ \hline \hline
(i-1)& $\str_E \oplus \str_E$ & no singular  fibers & $p\ge 0$ \\ \hline
(i-2)& $\str_E \oplus \mathcal{L}$,  $\ord \mc L=m>1$ &$2\times {}_m\mathrm{I}_0$   & $p\ge 0$ \\ \hline
(i-3)& $\str_E \oplus \mathcal{L}$,  $\ord \mc L=\infty$ & &$p\ge 0$\\ \hline
(i-4)& indecomposable &  & $p=0$ \\ \hline
(i-5)& indecomposable & ${}_p\mathrm{I}_0$ $($a wild fiber$)$   & $p>0$ \\ \hline
\end{tabular} 
\end{center}
\item
Suppose that $e=-1$ and $p\ne 2$. Then, $S$ has an elliptic fibration with $3$ singular fibers of type ${}_2\mathrm{I}_0$.
\end{enumerate}
\end{thm}

Maruyama also considers the condition that elliptic ruled surfaces have an elliptic fibration 
\cite[Theorem 4]{Ma71}, in terms of elementary transformations of ruled surfaces.

\begin{rem}\label{rem:lambda}
Let $C_0$ be a section of $f$ satisfying $\mc{O}_S(C_0)\cong \mc{O}_{\PP (\mc{E})}(1)$ (see \cite[p.~373]{Ha77}),
$F$ be a general fiber of $\pi$, and $F_f$ a fiber of $f$.
Then \cite[V.~Corollary 2.11]{Ha77} tells us that 
$$
K_S\equiv -2C_0-eF_f,
$$
and by the canonical bundle formula of elliptic fibrations, we have
$$
K_S\equiv -\frac{2}{m}F
$$
in the case (i-2), and 
$$
K_S\equiv -\frac{1}{2}F
$$
in the case (ii).
Then, we can see that $F\cdot F_f=m$ (resp.~$F\cdot C_0=2$), and hence, 
 we have $\lambda_S=m$  (resp.~$\lambda_S=2$) in (i-2) (resp.~in (ii)).
\end{rem}




\section{Proof of Theorem \ref{thm:FMpartner}}\label{subsec:proof}
We give the proof of Theorem \ref{thm:FMpartner} in the last of this section. Before giving the proof, 
we need several claims.

Let us take a cyclic group $G=\mathbb{Z}/m\mathbb{Z}$ for an integer $m>1$ and a generator $g$ of $G$.
For integers $i\in (\Z /m\Z)^*$, 
define representations 
$$
\rho_{\PP ^1}  \colon G\to \Aut (\PP ^1)  \quad 
\mbox{ as }\quad 
\rho_{\PP ^1} (g)(y)=\zeta y,
$$
and
$$
\rho_{F,i} \colon G\to \Aut (F) \quad  
\mbox{ as }\quad 
\rho_{F,i} (g)(x)=T_{i\cdot a}x,
$$
where $a$ is an element of ${}_mF$, $T_a$ is the translation by $a$
 and $\zeta$ is a primitive $m$-th root of unity in $\C$.  
Let us consider the diagonal action
\begin{equation}\label{eqn:action}
\rho _i(:=\rho_{F,i}\times \rho_{\PP ^1}) \colon G\to \Aut (F\times \PP ^1)  
\end{equation}
induced by $\rho_{\PP ^1}$ and $\rho_{F,i}$.
Set
$$
S_i:=(F\times \PP ^1)/_i G,
$$ 
the quotient of $F\times \PP ^1$ by the action $\rho_i$.
Then we have the following commutative diagram:
\begin{align}\label{arr:quotient}
\xymatrix{ 
F \ar[d]_q &   F\times \PP ^1    \ar[l]_{p_1} \ar[r]^{p_2} \ar[d]^{q_i}  & \PP ^1  \ar[d]^{q_{\PP^1}}  \\
E:=F/\Span{a}        &  S_i  \ar[l]_{\qquad f_i} \ar[r]^{\pi_i\qquad}        & \PP ^1/G\cong \PP ^1 
}
\end{align}
Here, every vertical arrow is the quotient morphism of the action of $G$.
We can readily see that  $f_i$ is a $\PP^1$-bundle and $\pi_i$ is an elliptic fibration. 
Note that the quotient morphism $q$ 
does not depend on the choice of $i$, and that
 the left square in \eqref{arr:quotient} is a fiber product.
We can also see that $\pi_i$ has exactly 
two multiple fibers of type $_m\mathrm{I}_0$ over the branch points $q_{\PP^1}(0),q_{\PP^1}(\infty)$ of $q_{\PP^1}$, 
and it fits into the case  (i-2) in Theorem \ref{thm:TU14}. Consequently, there is a line bundle
$\mc L_i\in {}_m\hat{E}$ 
such that 
$$
S_i \cong \PP (\mc O_E\oplus \mc L_i)
$$ 
for each $i$.
Furthermore, because the left square in \eqref{arr:quotient} 
is a fiber product, we have
$q^*\mc L_i=\mc O_F$, which implies that $\Span {\mc L_i}=\Ker \hat{q}$
for the dual isogeny $\hat{q}\colon \hat{E}\to \hat{F}$ of $q$. 
Therefore, the subgroup $\Span {\mc L_i}$ of $\hat{E}$ does not depend on the choice of $i$.
In particular, we have an inclusion
\begin{equation}\label{eqn:FPm0}
\bigl\{ S_i  \bigm | i\in (\Z/m\Z)^* \bigr\}/\cong \
\hookrightarrow
\ \{ \PP (\mc O_E\oplus \mc L_1^{i}) \bigm| i\in (\Z/m\Z)^* \}/\cong.
\end{equation}
We will see below that these sets actually coincide by checking their cardinality.
Let us start the following claim.


\begin{cla}\label{cla:isom_ruled}
Take a line bundle $\mc{L}\in{}_m\hat{E}$. 
For $i, j\in (\Z /m\Z)^*$,
$\PP(\mc{O}_E\oplus \mc{L}^{i})\cong \PP(\mc{O}_E\oplus \mc{L}^{j})$ 
if and only if 
there is a group automorphism $\psi_1\in \Aut_0 E$ such that $\psi_1^*\mc{L}\cong \mc{L}^{\pm i^{-1}j}$ holds.
Consequently,  the cardinality of the right hand side of \eqref{eqn:FPm0} is $\varphi(m)/|H_{\hat{E}}^{\mc{L}_1}|$.
\end{cla}

\begin{proof}
Since each of $\PP(\mc{O}_E\oplus \mc{L}^{i})$ and $\PP(\mc{O}_E\oplus \mc{L}^{j})$ has a unique $\PP^1$-bundle structure,
any isomorphism $\psi\colon \PP(\mc{O}_E\oplus \mc{L}^{i})\to \PP(\mc{O}_E\oplus \mc{L}^{j})$
induces an automorphism $\psi_1$ of $E$, which is compatible with $\psi$. 
We can see by \cite[II.~Ex.~7.9(b)]{Ha77} that $\psi_1$ satisfies the desired property.
The opposite direction also follows from [ibid.]. 
\end{proof}
We also have the following.


\begin{cla}\label{cla:S_iS_j}
For $i, j\in (\Z /m\Z)^*$,
$S_i\cong S_j$ if and only if there is a group automorphism $\phi_1\in \Aut_0 F$ such that $\phi_1(a)=(\pm i^{-1}j)\cdot a$ holds.
Consequently,  the cardinality of the left hand side of \eqref{eqn:FPm0} is $\varphi(m)/|H_F^a|$.
\end{cla}

\begin{proof}
Suppose that there is an isomorphism $\psi\colon S_i\to S_j$.
As in the proof of Claim \ref{cla:isom_ruled}, 
$\psi$ induces an automorphism $\psi_1$ of $E$ which is compatible with $\psi$.
It is also satisfied that the dual isogeny $\hat{\psi_1}$ 
preserves the subgroup $\Span{\mc{L}_1}$ of $\hat{E}$, and hence 
$\psi_1$ lifts an automorphism $\phi_1$ of $F\cong \hat{F}\cong E/\Span{\mc{L}_1}$.  
Since the left square in \eqref{arr:quotient} 
is a fiber product, 
$\psi$ lifts to an automorphism $\phi$ of $F\times \PP^1$. 
We can see that $\phi$ is of the form 
$
\phi_1 \times \phi_2
$
for some $\phi_2 \in \Aut \PP^1$. 
Since any translation on $F$ descends to a translation on $E$, replacing $\phi_1$ if necessary, 
we may assume that $\phi_1\in \Aut_0 F$. 
Since $\phi$ descends to $\psi$, it should satisfy 
$$
\phi\circ\rho_i(g)= \rho_j(g^k)\circ \phi 
$$
for any $g\in G$ and some $k\in\Z$. By observing the action on $\PP^1$, we know that $k=1$ or $m-1$, and moreover  
\begin{align*}
\phi_2(y)=
\begin{cases}
\lambda y \qquad &\text{ (in the case $k=1$)}\\
\lambda/y \qquad &\text{ (in the case $k=m-1$)}
\end{cases}
\end{align*}
for $y\in \PP ^1$ and some $\lambda \in \C^*$.
In the former case, we obtain that 
$
\phi_1(a)=(i^{-1}j)\cdot a
$ 
holds, and in the latter case,
 $\phi_1(a)=(-i^{-1}j)\cdot a$
holds. 
\end{proof}
For $m\le 3$, we can easily see from Claims \ref{cla:isom_ruled} and \ref{cla:S_iS_j}
that the both side of \eqref{eqn:FPm0} coincide.
And hence, suppose that $m>3$. Then, it follows from  
 Lemma \ref{lem:F=E}, Claims \ref{cla:isom_ruled} and \ref{cla:S_iS_j} that
the both side of \eqref{eqn:FPm0} coincide:
%
\begin{align}\label{ali:FPm}
&\bigl\{ S_i  \bigm | i\in (\Z/m\Z)^* \bigr\}/\cong \notag\\
=&\{ \PP (\mc O_E\oplus \mc L_1^{i}) \bigm| i\in (\Z/m\Z)^* \}/\cong.
\end{align} 
The cardinality of this set is 
$\varphi(m)/|H_{\hat{E}}^{\mc{L}_1}|$.


\begin{cla}\label{cla:FM_partner}
In the above notation,
$S_i\cong J_{S_1}(i)$ for all $i$ with $i\in (\Z/m\Z)^*$. 
\end{cla}

\begin{proof}
Take an element $j\in (\Z/m\Z)^*$ such that 
$
ij=1.
$
Henceforth, we identify $F$ and $\hat{F}$ as group schemes by \eqref{eqn:dual}.
For the normalized Poincare bundle $\mc P_0$  given in \S \ref{subsec:technical},
we define 
$$
\mc P:=\mc P_0\otimes p_1^*\mc O_F(iO)\otimes p_2^*\mc O_F(jO).
$$ 
Here, we regard elements $i,j\in (\Z/m\Z)^*$ as integers satisfying $1\le i,j \le m-1$.
Then the line bundle $\mc P$ satisfies
\begin{equation}\label{eqn:P|_y}
\mc P|_{x\times F}\cong \mc O_F(x+(j-1)O)
\mbox{ and }
\mc P|_{F\times y}\cong \mc O_F(y+(i-1)O)
\end{equation}
for any $x,y\in F$.
Let us consider the commutative diagram:
\[ 
\xymatrix{ 
x\times F \ar@{^{(}->}[r] \ar[d]_{T_{i\cdot a}} &   F\times F    \ar[d]_{T_a\times T_{i\cdot a}} 
 & F\times y  \ar[d]_{T_a}  \ar@{_{(}->}[l]  \\
(x\oplus a)\times F\ar@{^{(}->}[r] &  F\times F  & F\times (y\oplus i\cdot a)  \ar@{_{(}->}[l] 
}
\]
Here, the left vertical morphism is defined by the composition of morphisms
$$
 x\times F\cong F \stackrel{T_{i\cdot a}}\longrightarrow  F\cong (x\oplus a)\times F
$$
and  similarly, the right vertical arrow is also defined by $T_a$. 
Now we have 
\begin{align*}
((T_a\times T_{i\cdot a})^*\mc P)|_{F\times y}
&\cong T_a^*(\mc P|_{F\times (y\oplus i\cdot a)})\\
&\cong\mc P|_{F\times (y\oplus i\cdot a)}\otimes \mc O_F(a-O)^{-i}\\
&\cong\mc O_F(y+i a-O) \otimes \mc O_F(a-O)^{-i}\\
&\cong\mc O_F(y+(i-1)O)\\
&\cong\mc P|_{F\times y}.
\end{align*}
Using $\ord a=m$, we also have 
\begin{align*}
((T_a\times T_{i\cdot a})^*\mc P)|_{x\times F}
&\cong T_{i\cdot a}^*(\mc P|_{(x\oplus a)\times F})\\
&\cong\mc P|_{(x\oplus a)\times F}\otimes \mc O_F(i a-iO)^{-j}\\
&\cong\mc O_F(x+a+(j-2)O) \otimes \mc O_F(ia-iO)^{-j}\\
&\cong\mc O_F(x+(j-1)O)\\
&\cong\mc P|_{x\times F}.
\end{align*} 
Hence, we obtain 
$(T_a\times T_{i\cdot a})^*\mc P\cong \mc P$
by \cite[III.~Ex.~12.4]{Ha77}.
Let us define $\Delta _{\PP ^1}(\cong \PP^1)$ to be the diagonal in $\PP ^1 \times \PP ^1$.
For the projection  
$$p_{1}\colon (F \times F) \times \Delta _{\PP ^1} \to F\times F,$$
define a sheaf
$$\mc U:=p_{1}^*\mc P.$$
We regard $\mc{U}$ as a sheaf on $(F\times \PP^1)\times (F\times \PP^1)$. 
Then for any $g\in G$, we have 
\begin{align*}
((\rho_1 (g)\times \rho_i (g))^*\mc U)|_{x \times \PP^1 \times y \times \PP^1}
&\cong (\rho_{\PP ^1} (g)\times  \rho_{\PP ^1} (g))^*(\mc U|_{(x\oplus a)\times \PP^1 \times (y\oplus i\cdot a)\times \PP^1})\\
&\cong \rho_{\PP ^1}(g)^*\mc O_{\Delta _{\PP ^1}}\\
&\cong \mc O_{\Delta _{\PP ^1}}\\
&\cong \mc U|_{x \times \PP^1 \times y \times \PP^1}
\end{align*}
for $x,y\in F$, and 
\begin{align*}
((\rho_1 (g)\times \rho_i (g))^*\mc U)|_{F \times z\times F \times z}
&\cong (T_a\times T_{i\cdot a})^*(\mc U|_{F \times \zeta z\times F \times \zeta z})\\
&\cong (T_a\times T_{i\cdot a})^* \mc P\\
&\cong \mc P\\
&\cong \mc U|_{F \times z\times F \times z}
\end{align*}
for any $z\in \PP^1$, and note that
both of 
$
((\rho_1 (g)\times \rho_i (g))^*\mc U)|_{F \times z_1\times F \times z_2}
$ 
and
$
\mc U|_{F \times z_1\times F \times z_2}
$
are zero  for $z_1\ne z_2\in \PP^1$, since $\Supp \mc{U}= (F \times F) \times \Delta _{\PP ^1}$.
Then, by \cite[III.~Ex.~12.4]{Ha77}
 we can check  
$
(\rho_1 (g)\times \rho_i (g))^*\mc U\cong \mc U,
$
equivalently
$$
(\rho_1 (g)\times  {\id}_{F\times \PP^1})^*\mc U\cong ( {\id}_{F\times \PP^1}\times \rho_i (g^{-1}))^*\mc U.
$$
This implies that 
$$
(q_{1}\times {\id}_{F\times \PP^1})_*\mc{U}\cong ({\id}_{S_1}\times \rho_i(g^{-1}))^* 
(q_{1}\times {\id}_{F\times \PP^1})_*\mc{U},
$$
that is, the sheaf
$
 (q_{1}\times {\id}_{F\times \PP^1})_*\mc{U}
$
is $G$-invariant with respect to the diagonal action of $G$ on $S_1\times (F\times \PP^1)$, 
induced by the trivial action on $S_1$ and $\rho_i$ on $F\times \PP^1$.
Since $G$ is cyclic, we can conclude that
$
 (q_{1}\times {\id}_{F\times \PP^1})_*\mc{U}
$
is $G$-equivariant, and hence there is a coherent sheaf  $\mc{U}'$ on $S_1\times S_i$ such that
\begin{equation}\label{eqn:descent}
 (q_{1}\times {\id}_{F\times \PP^1})_*\mc{U}\cong  ( {\id}_{S_1}\times q_i)^*\mc{U}'.
\end{equation}
For $x\times z\in F\times \PP^1$, we have $\mc{U}|_{F\times z\times x\times z}\cong \mc{P}|_{F\times x}$, which is 
a  line bundle of degree $i$ on $F$ by \eqref{eqn:P|_y}.
The isomorphism (\ref{eqn:descent}) yields 
$$
\mc{U}'|_{S_1\times q_i(x\times z)}\cong ((q_{1}\times {\id}_{F\times \PP^1})_*\mc{U})|_{S_1\times (x\times z)}
\cong (q_{1}\times {\id}_{F\times \PP^1})_* (\mc{U}|_{ (F\times \PP^1)\times (x\times z)}).
$$
Here, the second isomorphism follows from \cite[Lemma 1.3]{BO95} and the smoothness of $q_1$.
Since $\mc{U}|_{(F\times \PP^1)\times (x\times z)}$ is actually a sheaf on $F\times z \times x\times z$ and 
 the restriction $q_1|_{F\times z}$  is isomorphic for $z\in \PP^1\backslash \{0,\infty\}$,
 $\mc{U}'|_{S_1\times {q_i(x\times z)}}$ is 
also a line bundle of degree $i$ on $F_z\times q_i(x\times z)$ for such $z$. 
Here, $F_z(\cong F)$ is a fiber of $\pi_1$ over the point $q_{\PP ^1}(z)$. 
Then, by the universal property of $J_{S_1}(i)$,
there is a morphism between the open subsets of $S_i$ and $J_{S_1}(i)$ over $\PP^1\backslash\{q_{\PP^1}(0),q_{\PP^1}(\infty)\}$.
Since $\mc{U}'|_{S_1\times q_i(x\times z)}\not\cong\mc{U}'|_{S_1\times q_i(y\times z)}$ on $F_z$ for $x\ne y\in S_i$, this morphism is injective, and hence  
$S_i$ and $J_{S_1}(i)$ are birational over $\PP^1$.
Then, \cite[Proposition III.~8.4]{BHPV} implies the result.
\end{proof}

Now, we obtain  the following.


\begin{prop}\label{prop:FM_P(O)}
Let $E$ be an elliptic surface, and define $S$ to be an elliptic ruled surface $\PP (\mc{O}_E\oplus \mc{L})$
for a line bundle $\mc{L}\in{}_m\hat{E}$ for $m>0$. 
Then we have
$$
\FM (S)=\{\PP (\mc{O}_E\oplus \mc{L}^i)\mid i\in (\Z/m\Z)^*\}/\cong.
$$
This set consists of $\varphi(m)/|H_{\hat{E}}^{\mc{L}}|$ elements. In the case $m>3$, the cardinality $|H_{\hat{E}}^{\mc{L}}|$ is $2,4$ or $6$, 
depending on the choice of $\hat{E}$ and $\mc{L}$. 
\end{prop}

\begin{proof}
The first statement is a direct consequence of  Theorem \ref{BMelliptic}, 
the equation \eqref{ali:FPm} and Claim \ref{cla:FM_partner}.
The second  is a direct consequence of Claim \ref{cla:isom_ruled}. We can compute the cardinality of 
$H_{\hat{E}}^{\mc{L}}$ by Lemmas \ref{lem:j-invariant} and \ref{lem:F=E}.

\end{proof}

We are in a position to show Theorem \ref{thm:FMpartner}.\\

\noindent
{\it Proof of Theorem \ref{thm:FMpartner}.}
The condition $|\FM(S)|\ne 1$ implies that 
$S$ has an elliptic fibration  $\pi\colon S\to \PP^1$ (see \cite{BM01}).  
Hence, either of the cases (1-i), (1-ii) or (ii) in Theorem \ref{thm:TU14} occurs (recall that we work over $\C$).
In each case, we see from Remark \ref{rem:lambda} that $\lambda_S=1,m$ and $2$ respectively. 
It follows from Claim \ref{cla:lambda} that $S$ actually fits into the case (1-ii) with $m>4$. 
Now set
$S=\PP (\mc{O}_E\oplus \mc{L})$  
for some $\mc{L}\in{}_m\hat{E}$ for some $m>4$.
Then the assertion follows from Proposition \ref{prop:FM_P(O)}.
\qed\\


\section{Further questions}\label{sec:autoeq}
\subsection{Autoequivalences}
Let $S:=\PP(\mc{O}_E\oplus \mc{L})$ be an elliptic ruled surface 
with non-trivial Fourier--Mukai partners, where
$E$ is an elliptic curve, and $\mc{L}\in {}_mE$ for some $m>4$, as in Theorem \ref{thm:FMpartner}. 
Then, the group $H_S$ defined in \S \ref{subsec:bridgeland} coincides with the group $H_{\hat{E}}^{\mc{L}}$, by the results in \S 
\ref{subsec:proof}.
Note  that  there are no $(-2)$-curves on $S$, and hence no twist functors associated with $(-2)$-curves
appears in  $\Auteq   D(S)$.
Moreover, we can see that the $\PP^1$-bundle $f\colon S\to E$ has two sections $C_0$ and $C_1$,
and $mC_0$ and $mC_1$ are the multiple fibers of $\pi$. 
We can also check that 
$$
\Span{ \mathcal{O}_S(D)\mid  D\cdot F=0}=\Span{\mathcal{O}_S(C_0),\mathcal{O}_S(C_1) }
$$
in $\Pic (S)$, where $F$ is a smooth fiber of $\pi$.
Therefore, by the main theorem of \cite{Ue15},
we have the following short exact sequence:
\begin{align*}
1\to 
\Span{\otimes \mathcal{O}_S(C_0),\otimes\mathcal{O}_S(C_1)}
&\rtimes \Aut S \times \Z[2] \to
\Auteq   D(S) \notag\\
&\to
\bigl\{ \begin{pmatrix}
c& a\\
d& b   
\end{pmatrix}\in \Gamma_0(m) \bigm| b\in H_{\hat{E}}^{\mc{L}}                       \bigr\}
\to 1.
\end{align*}
Here for an integer $b$, coprime with $m$, we again denote by $b$ 
 the corresponding element in $H_{\hat{E}}^{\mc{L}}(\subset (\Z/m\Z)^*)$, 
and $\Gamma_0(m)$ is the congruence subgroup of $\SL (2,\Z)$, defined in \cite{Ue15}.

Since other ruled surfaces with an elliptic fibration has no non-trivial Fourier--Mukai partners,
the description of their autoequivalence groups is directly given by  \cite{Ue15}.

For elliptic ruled surfaces without elliptic fibrations, 
a description of the autoequivalence group will be given in the forthcoming paper  \cite{Ue}.

\subsection{Positive characteristic}
The proof of Theorem \ref{thm:FMpartner} does not work over positive characteristic fields.
We finish this section to raise the following:

\begin{prob}\label{prob:positive}
\begin{enumerate}
\item
In the notation of \S \ref{subsec:positive}, consider the case $e=-1$ and $p=2$. 
Study when $S$ has an elliptic fibration, and if it has,
study the singular fibers of the fibration.   
\item
Describe the set $\FM(S)$ for elliptic ruled surfaces $S$ over a positive characteristic field.
\end{enumerate}
\end{prob}

In \cite[Theorem 4]{Ma71}, Maruyama states that in the case $e=-1$ and $p=2$, $S$ ($\mathbf{P}_1$ in his notation)
has an elliptic fibration. 
But it seems to the author that he gave no proof of this statement. See also \cite[Remark 7]{Ma71}.

Furthermore, in the case (i-5) in Theorem \ref{thm:TU14}, if $p\ge 5$,
$S$ may have non-trivial Fourier--Mukai partners, since $\lambda_S=p$ (we omit the proof of this fact here).
It is also an interesting question to describe $\FM(S)$ in this case. \cite[Examples 4.7, 4.8]{KU85} fit into the case (i-5). 
To show  Theorem \ref{thm:FMpartner}, the equality \eqref{ali:FPm} was a key.
The author believes that the description in  \cite[Examples 4.7, 4.8]{KU85} should be useful to describe $\FM(S)$ for 
$S$ in the case (i-5).


\noindent
Department of Mathematics
and Information Sciences,
Tokyo Metropolitan University,
1-1 Minamiohsawa,
Hachioji-shi,
Tokyo,
192-0397,
Japan 

{\em e-mail address}\ : \  hokuto@tmu.ac.jp

\begin{thebibliography}{BM01}

\bibitem[BO95]{BO95}
A.I. Bondal, D.O. Orlov, Semiorthogonal decomposition for algebraic varieties, alg-geom 9712029.

\bibitem[BHPV]{BHPV} Barth, Wolf P.; Hulek, Klaus; Peters, Chris A. M.; Van de Ven, Antonius, Compact complex surfaces. Second edition. Results in Mathematics and Related Areas. 3rd Series. A Series of Modern Surveys in Mathematics, 4. Springer-Verlag, Berlin, 2004. xii+436 pp.

\bibitem[Br98]{Br98}
T. Bridgeland, Fourier--Mukai transforms for elliptic surfaces. J. Reine Angew. Math. 498 (1998), 
115--133.



\bibitem[BM01]{BM01}
T. Bridgeland, A. Maciocia,
Complex surfaces with equivalent derived categories. Math. Z. 236 (2001), 677--697. 



\bibitem[Ha77]{Ha77}
R. Hartshorne, Algebraic Geometry, Springer--Verlag, Berlin Heidelberg New York, 1977.



\bibitem[KU85]{KU85}
T. Katsura, K. Ueno,  On elliptic surfaces in characteristic $p$, Math. Ann. 272, 291--330 (1985). 

\bibitem[Ka02]{Ka02}
Y. Kawamata, D-equivalence and K-equivalence. J. Differential Geom. 61 (2002), 147--171.

\bibitem[Ma71]{Ma71}
M. Maruyama,  On automorphism groups of ruled surfaces, J. Math. Kyoto Univ.  11  (1971), 89-–112. 



\bibitem[Or97]{Or97}
D. Orlov, Equivalences of derived categories and K3 surfaces. Algebraic geometry, 7. J. Math. Sci. (New York) 84 (1997), 1361--1381.
 




\bibitem[To11]{To11}
T. Togashi, "Daen fibration wo motsu seihyousuu no sensiki kyokumen ni tsuite"  (in Japanese), Master's Thesis,  Tokyo Metropolitan University, (2011).

\bibitem[Ue04]{Ue04}
H. Uehara, An example of Fourier-Mukai partners of minimal elliptic surfaces. Math. Res. Lett. 11 (2004), no. 2-3, 371--375.

\bibitem[Ue11]{Ue11}
H. Uehara, A counterexample of the birational Torelli problem via Fourier-Mukai transforms. J. Algebraic Geom. 21 (2012), no. 1, 77--96.

\bibitem[Ue15]{Ue15}
H. Uehara, Autoequivalences of derived categories of elliptic surfaces with non-zero Kodaira dimension,  arXiv:1501.06657.

\bibitem[Ue]{Ue}
H. Uehara, Autoequivalences of derived categories of surfaces with non-torsion canonical divisor. (in preparation)
\end{thebibliography}
\end{document}